\RequirePackage{rotating}
\documentclass{amsart}  

\usepackage[foot]{amsaddr}

\usepackage[usenames, dvipsnames]{color}
\usepackage{amsfonts}
\usepackage{amsthm}
\usepackage{amsmath}
\usepackage{amsfonts}
\usepackage{latexsym}
\usepackage{amssymb}
\usepackage{amscd}
\usepackage[latin1]{inputenc}
\usepackage{verbatim}
\usepackage{enumerate}
\usepackage{enumitem}
\usepackage{graphicx}
\usepackage{xcolor}
\usepackage{tikz}
\usepackage{caption}
\usepackage{adjustbox}
\usepackage{algorithmic}
\usepackage{epstopdf}
\usepackage{fancybox}
\usepackage{booktabs} 
\usepackage{expl3}
\usepackage{l3keys2e}
\usepackage{xparse}
\usepackage{blkarray}
\usepackage{pst-node}
\usepackage{animate}
\usepackage{float}
\usepackage{booktabs}
\usepackage{makecell}
\usepackage{systeme}
\usepackage{placeins}
\usetikzlibrary{matrix}
\usetikzlibrary {positioning}
\usetikzlibrary{arrows,shapes}
\usetikzlibrary{trees}
\usetikzlibrary{shapes.geometric}
\usetikzlibrary{calc,shapes.callouts,shapes.arrows}
\usetikzlibrary{graphs}
\usetikzlibrary{positioning, arrows}

\definecolor{mybluegreen}{rgb}{0.1, 0.55, 0.35}
\usepackage{caption}
\usepackage{array,booktabs}
\usepackage{siunitx}
\usepackage{rotating}
\usepackage[spaces,hyphens]{url}
\usepackage[colorlinks,allcolors=blue]{hyperref}

\usepackage[math]{cellspace}
\newtheorem{theorem}{Theorem}[section]

\newtheorem{example}[theorem]{Example}

\newtheorem{definition}[theorem]{Definition}
\newtheorem{conj}{Conjecture}[section]

\theoremstyle{definition}

\newcommand{\Sym}{\mathrm{Sym}}

\definecolor{mybluegreen}{rgb}{0.1, 0.55, 0.35}
\newcolumntype{?}{!{\vrule width 2pt}}
\newcolumntype{M}[1]{>{\centering\arraybackslash}m{#1}}
\newcolumntype{P}[1]{>{\centering\arraybackslash}p{#1}}
\newcolumntype{N}{@{}m{0pt}@{}}
\title[EKR Theorem for set-wise $2$ and $3$-intersecting families of perfect matchings]{An extension of the Erd\H{o}s-Ko-Rado theorem to set-wise $2$-intersecting families of perfect matchings}

\author[M.~N.~Shirazi]{Mahsa N. Shirazi} \email[M.~N.~Shirazi]{mahsa.nasrollahi@gmail.com}
\address{Department of Mathematics and Statistics, University of Regina, Regina, SK, S4S 0A2,
Canada}

\date{\today}

 		
\begin{document}

%______________________________________________________________%
%______________________________________________________________%

\begin{abstract}  
Two perfect matchings $P$ and $Q$ of the complete graph on $2k$ vertices are said to be set-wise $t$-intersecting if there exist edges $P_{1}, \cdots, P_{t}$ in $P$ and $Q_{1}, \cdots, Q_{t}$ in $Q$ such that the union of edges $P_{1}, \cdots, P_{t}$ has the same set of vertices as the union of $Q_{1}, \cdots, Q_{t}$ has. In this paper we prove an extension of the famous Erd\H{o}s-Ko-Rado (EKR) theorem to set-wise $2$-intersecting families of perfect matching on all values of $k$, and we conjecture similar statement for all $t\geq 2$.
\end{abstract}

%-----------------------------------------
\maketitle
%______________________________________________________________%
%______________________________________________________________%
\section{Introduction and Preliminaries}\label{sec:background}

The Erd\H{o}s-Ko-Rado (EKR) theorem states that if $\mathcal{F}$ is a $t$-intersecting family of $k$-subsets of $\{1,2,\ldots, n\}$, then $\binom{n-t}{k-t}$ is a tight upper bound on the size of $\mathcal{F}$ with $n$ sufficiently large~\cite{EKR.1961}. This famous theorem, proved in 1961, has been extended and modified into many versions and extensions to different mathematical objects such as permutations \cite{EKRpermGM, EKRpermK}, uniform set-partitions \cite{MM, KMB}, $t$-designs  \cite{tDesign}, vector spaces \cite{EKRvecSpace}, and perfect matchings \cite{2int, GMP, L}.\\

The focus of this work is on perfect matchings.  A \textsl{perfect matching} in a graph $G$, is a set of edges by which every vertex is covered exactly once. The number of perfect matching in the complete graph $K_{2k}$ is 
%----------------------------------
\begin{equation*}
\frac{1}{k!}\binom{2k}{2}\binom{2k-2}{2}\cdots \binom{2}{2} = (2k-1)(2k-3)(2k-5)\cdots 1.
\end{equation*}
%----------------------------------
For any positive integer $k$ define 
%----------------------------------
\begin{equation*}
(2k-1)!!:=(2k-1)(2k-3)(2k-5)\cdots 1,
\end{equation*}
%----------------------------------
Therefore, the number of perfect matchings in $K_{2k}$ is $(2k-1)!!$. Two perfect matchings are said to be \textsl{$t$-intersecting} if they have at least $t$ edges in common. In \cite{2int}, using representation theory, particular properties of association schemes, and eigenvalue techniques Fallat, Meagher, and Shirazi proved the following theorem:
%----------------------------------
\begin{theorem}\cite{2int}
The size of the largest set of $2$-intersecting perfect matchings in the complete graph $K_{2k}$ with $k\geq 3$ is $(2k-5)!!$. Further, if $S$ is a set of $2$-intersecting perfect matchings its the characteristic vector is a linear combination of the characteristic vectors of the canonically 2-intersecting sets of perfect matchings.
\end{theorem}
%---------------------------------
In \cite{Ellis}, Ellis used the concept of set-wise intersecting families for  permutations. A potentially fruitful direction beyond Theorem 1.1 is to generalize the definition of intersection to set-sise intersection for perfect matchings; this is the motivation for this paper. In this work we prove an extension of the famous Erd\H{o}s-Ko-Rado theorem to set-wise $2$-intersecting families of perfect matchings for the complete graph $K_{2k}$. 
%---------------------------------
\begin{definition}
For $t\leq \lfloor \frac{k}{2} \rfloor$, two perfect matchings $P$ and $Q$ of a graph on $2k$ vertices are said to be \textsl{set-wise $t$-intersecting} if there exist edges $P_{1}, \ldots, P_{t}$ in $P$ and $Q_{1}, \ldots, Q_{t}$ in $Q$ such that the union of edges $P_{1}, \ldots, P_{t}$ has the same set of vertices as the union of $Q_{1}, \ldots, Q_{t}$. 
\end{definition}
%---------------------------------
Note that in this definition, set-wise $t$-intersection implies set-wise $(k-t)$-intersection; hence we only consider $t\leq \lfloor \frac{k}{2} \rfloor$. The following example depicts the difference between $2$-intersecting and set-wise $2$-intersecting perfect matchings in the complete graph $K_{2k}$. Every pair of $t$-intersecting perfect matchings is also set-wise $t$-intersecting perfect matching, though the converse is not true. \\
%---------------------------------
\begin{example}
In Figure \ref{fig:K8}, each set of colored egges with the same color forms a perfect matching in the complete graph $K_{2k}$. The orange perfect matching and the green perfect matching are both $2$-intersecting and set-wise $2$-intersecting, while the orange and blue perfect matchings are set-wise $2$-intersecting, but not $2$-intersecting.
\begin{figure}[H]
\includegraphics[scale=0.3]{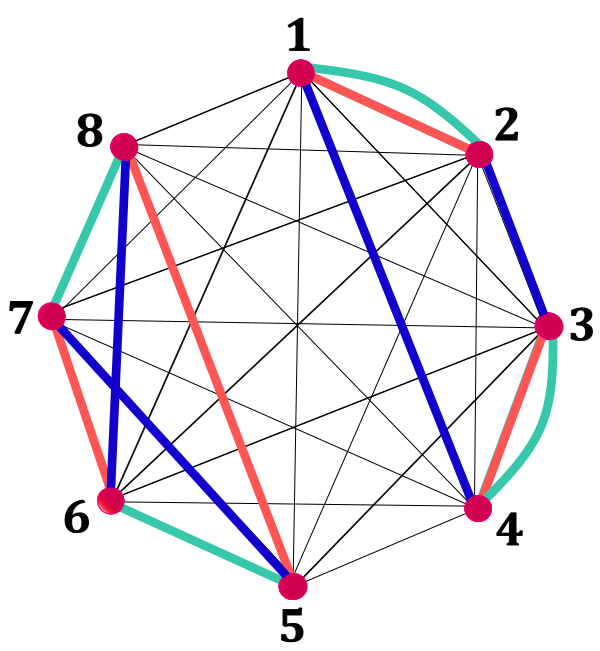}
\caption{$2$-intersecting perfect matchings vs set-wise $2$-intersecting perfect matchings in $K_{8}$\label{fig:K8}}
\end{figure}
\end{example}
%---------------------------------
Let $\mathcal{S}_{t}(2k)$ denote the set of perfect matchings in $K_{2k}$ in which the vertices $1,2,\dots, 2t$ are covered with exactly $t$ edges. Then, the set $\mathcal{S}_{t}(2k)$ forms a family of set-wise $t$-intersecting perfect matchings, and it is easy to check that the cardinality of this set is $(2t-1)!!(2k-2t-1)!!$. The set $\mathcal{S}_{t}(2k)$ is an orbit of the group $Sym(2t)\times Sym(2k-2t)$.  We propose the following conjecture
%---------------------------------
\begin{conj}
The size of the largest set of set-wise $t$-intersecting perfect matchings in the complete graph $K_{2k}$ with $k\geq 2t$ is $(2t-1)!!(2k-2t-1)!!$ and a set with maximum size is an orbit of the Young subgroup $Sym(2t)\times Sym(2k-2t)$.
\end{conj}
%---------------------------------
In this paper we are going to prove the aforementioned conjecture for $t$ being 2. First, in Section \ref{sec:PM Asso}, we provide some necessary background on the perfect matching association scheme, and its character table; we will calculate several eigenvalues in the character table for all values of $k\geq 7$. We also, define a graph $N_{t}(2k)$ for which finding the size of the largest cocliques is equivalent to finding the size largest set of set-wise $t$-intersecting perfect matchings. The eigenvalues of this graph can be derived from the eigenvalues of the graphs in the perfect matching association scheme. Using eigenvalue techniques, in Section \ref{sec:EKR} we construct a weighted adjacency matrix for the graph $N_{2}(2k)$ and we show that for this matrix, the ratio bound (see Theorem \ref{RatioBound}) holds with equality. A similar approach is used in \cite{2int} which was inspired by the eigenvalue techniques in \cite{GMP} and \cite{W}. In Section \ref{sec:set-wise3int} we construct a weighted adjacency matrix for for the graph $N_{3}(2k)$, and conjecture that the ratio bound holds with equality for this matrix. We conclude this paper with some open problems and possible future directions. In the appendix we provide the diagonals of several adjacency matrices for some graphs in the perfect matching association scheme which were used to calculate the desired eigenvalues in their general forms.

%______________________________________________________________%
%______________________________________________________________%
\section{Perfect Matching Association Scheme}\label{sec:PM Asso}

In this section, we briefly introduce the perfect matching association scheme. For more details see \cite[Section 2.2]{2int} and \cite[Sections 3, 15.4]{GMB}. We denote an even integer partition of $2k$ by $\lambda = [\lambda_{1}, \lambda_{2},\dots, \lambda_{\ell}]$, where $2k=\sum_{i=1}^{\ell}\lambda_{i}$ and for all $1\leq i\leq \ell$, $\lambda_{i}$ is an even integer, and $\lambda_{i}\geq \lambda_{i+1}$, which is denoted by $\lambda \vdash 2k$. Taking the union of two perfect matchings in $K_{2k}$ produces disjoint even cycles and the length if these cycles construct an even partition of $2k$, called the \textit{shape} of their union. Note that the union of parallel edges is a cycle of length $2$.\\

For an even integer partition $\lambda$ on the set $\{1,\cdots, 2k \}$, define a matrix $A_{\lambda}$ in which rows and columns are indexed by perfect matchings of $K_{2k}$. In $A_{\lambda}$ the $(P,Q)$-entry is 1 if the union of the perfect matchings $P$ and $Q$ has shape $\lambda$, and 0 otherwise. Then the set $\mathcal{A}=\{ A_{\lambda} \, | \, \lambda \vdash 2k\}$ forms a symmetric association scheme which is known as the \textsl{perfect matching association scheme}~\cite[Section 15.4]{GMB}, and the set $\mathcal{A}$ forms an orthogonal basis of the matrix algebra $\mathbb{C}[\mathcal{A}]$ of the matrices in $\mathcal{A}$, called the \textsl{Bose-Mesner algebra} of this association scheme. The matrices $A_{\lambda_{i}}$ are symmetric and commutative, therefore, they are simultaneously diagonalizable.\\

The symmetric group $\Sym(2k)$ acts transitively on the set of perfect matchings of the complete graph $K_{2k}$, and the wreath product $\Sym(2) \wr \Sym(k)$ is isomorphic to the group that stabilizes a perfect matching. So the action of $\Sym(2k)$ on perfect matchings is equivalent to its action on the cosets of $\Sym(2k) / \left( \Sym(2) \wr \Sym(k) \right)$. Let $V$ be the vector space of the vectors indexed by perfect matchings of $K_{2k}$. Then $V$ corresponds to a $\Sym(2k)$-module. This module can be expressed as sum of irreducible representations of $\Sym(2k)$-modules. It is known that all irreducible representations of $\Sym(2k)$ correspond partitions and all representation modules correspond to even integer partitions of $2k$; therefore, $V$ is the sum of irreducible $\Sym(2k)$-modules corresponding to even integer partitions of $2k$. Thus, the common eigenspaces in the perfect matching association scheme correspond to even integer partitions. This means that both classes and the common eigenspace in this association scheme correspond to even integer partitions of $2k$.

%----------------------------------------- 
\subsection{Perfect Matching Derangement Graph and its Eigenvalues}

In this part, we first introduce some basic definitions and notations, and then define the perfect matching derangement graph $N_{t}(2k)$ for which; later in Section \ref{sec:EKR}; we show the ratio bound holds with equality for $t$ being $2$.\\

A clique in a graph $X$ is a set of vertices such that any two vertices are adjacent, and a coclique is a set of vertices in which no two vertices are adjacent. We denote the size of a maximum clique and a maximum coclique of the graph $X$ by $w(X)$ and $\alpha(X)$, respectively. The adjacency matrix $A(X)$ of the graph $X$ is a matrix in which the rows and columns are indexed by vertices and $(i,j)$-entry is $1$ if $i\sim j$, and 0 otherwise. A \textsl{weighted adjacency matrix} $B(X)$ of the graph $X$ is a symmetric matrix in which the rows and columns are indexed by the vertices and the $(i, j)$-entry may be non-zero if $i\sim j$ and is 0 otherwise. The \textsl{eigenvalues} of $X$ refer to the eigenvalues of its (weighted) adjacency matrix. For a matrix $A_{\lambda}$ in the perfect matchings association scheme, define the graph $X_{\lambda}$ such that $A_{\lambda}$ is its adjacency matrix. So a graph $X$ in this scheme is a graph with $A(X)\in \mathbb{C}[\mathcal{A}]$. 
%We use $\mathbf{1}$ to denote the all-ones vector, for any $d$-regular graph, the all-ones vector is an eigenvector with eigenvalue $d$.
%-----------------------------
\begin{definition}
For $t\leq \lfloor \frac{k}{2} \rfloor$, define the graph $N_{t}(2k)$ to be the graph with perfect matchings of the complete graph $K_{2k}$ as its vertices. In this graph, two perfect matchings are adjacent if there is no partition of $2t$ as a sub-partiton of their shape.
\end{definition}
%-----------------------------
For example, in $N_{3}(2k)$, two perfect matchings on $2k$ vertices are adjacent if there is no 6-cycle, no 4 and a 2-cycle, no  three 2-cycles in their intersection. Denote the adjacency matrix of $N_{t}(2k)$ by $A_{t}(2k)$. Clearly, the graph $N_{t}(2k)$ is a graph in the perfect matching association scheme and we have that
%----------------------------
\begin{equation}\label{eq:AdjMat}
A_{t}(2k) = \sum_{\lambda \vdash 2k}A_{\lambda},
\end{equation}
%----------------------------
where partitions $\lambda$ contain no partition of $[2t]$ as as a sub-partition. So the eigenvalues of $N_{t}(2k)$ are the sum of the eigenvalues of the matrices $A_{\lambda}$ in (\ref{eq:AdjMat}). Similarly, a weighted adjacency matrix $B_{t}(2k)$ of the graph $N_{t}(2k)$ is in $\mathbb{C}[\mathcal{A}]$,
%----------------------------
\begin{equation}\label{eq:W.AdjMat}
B_{t}(2k) = \sum_{\lambda \vdash 2k}a_{\lambda}A_{\lambda},
\end{equation}
%----------------------------
where $A_{\lambda}\in \mathcal{A}$, partitions $\lambda$ have no partition of $[2t]$ as a sub-partition and $a_{\lambda} \in \mathbb{C}$. The eigenvalue $\theta_{\mu}$ of $B_{t}(2k)$ belonging to the module $\mu$ is
\[
\theta_{\mu} = \sum_{\lambda \vdash 2k}a_{\lambda,}\theta_{\lambda}^{\mu},
\]
where, $\theta_{\lambda}^{\mu}$ is the eigenvalue of $A_{\lambda}\in \mathcal{A}$ corresponding to the $\mu$-module; and coefficients $a_{\lambda}$ are the same as the in (\ref{eq:W.AdjMat}). A coclique in the graph $N_{t}(2k)$ is a set of perfect matchings for which every two have a partiton of $[2t]$ in their intersection, in other words they are set-wise $t$-intersecting. So finding the largest size of a set-wise $t$-intersecting perfect matching on $2k$ vertices is equivalent to finding $\alpha(N_{t}(2k))$. The set $\mathcal{S}_{t}(2k)$ defined in Section \ref{sec:background} is a coclique in $N_{t}(2k)$ of size $(2t-1)!!(2k-2t-1)!!$. For $t=2$, our goal is to show that $\mathcal{S}_{t}(2k)$ is actually a maximum coclique in $N_{t}(2k)$, and so represents a family of set-wise $t$-intersecting perfect matching of the maximum size. For an arbitrary graph $X$, finding $\alpha(X)$ is an NP-hard problem, but the following theorem presents a classical tight upper bound for $\alpha(X)$.
%----------------------------
\begin{theorem}[Delsarte-Hoffman bound]\cite[p. 31]{GMB}\label{ratioBound}
Let $B(X)$ be a weighted adjacency matrix for a graph $X$ on vertex set $V(X)$. If $B(X)$ has constant row sum $d$ and least eigenvalue $\tau$, then 
\begin{equation*}\label{RatioBound}
\alpha(X)\leq \frac{|V(X)|}{1-\frac{d}{\tau}}.
\end{equation*}
If equality holds for some coclique $S$ with characteristic vector $\nu_{S}$, then 
\begin{equation*}
\nu_{S}-\frac{|S|}{|V(X)|}\mathbf{1}
\end{equation*}
is an eigenvector with eigenvalue $\tau$.
\end{theorem}
%--------------------------
This bound is also called the \textsl{ratio} bound. In the remainder if this paper we define the coefficients $a_{\lambda}$ in (\ref{eq:W.AdjMat}), for which the ratio bound holds with equality for $B_{t}(2k)$, and hence for $N_{t}(2k)$; when $t =2$. This nice approach first was developed by Wilson in \cite{W}. Later it was used in \cite{2int, GMP, L} to prove extensions of the EKR theorem to $t$-intersecting families of perfect matchings. Therefore, our approach is to show that
$1-\frac{d}{\tau}=\frac{\mid V(N_{t}(2k)) \mid}{\mid \mathcal{S}_{t}(2k) \mid}$. To verify this, we need to find the row sum and the least eigenvalue of the matrix $B_{t}(2k)$.
%-----------------------------------------
\subsection{The Character Table of the Perfect Matching Association Scheme}

As we mentioned previously, if we have the eigenvalues of matrices $A_{\lambda}$ in (\ref{eq:W.AdjMat}), then we can derive the eigenvalues of $B_{t}(2k)$. However, we don't have the complete character table for the perfect matching association scheme when $k >40$, (see \cite{Mu, Sri1,Sri2} for more details). A nice approach to find part of the character table for general $k$ is to use the concept of \textsl{quotient graphs}. The set partition $\mu = [\mu_{1},\mu_{2},\dots, \mu_{i}]$ of the vertices in a graph $X$ is said to be \textsl{equitable} if for any vertex in $\mu_{k}$ the number of adjacent vertices in $\mu_{j}$ is only determined by $k$ and $j$. The quotient graph $X/ \mu$, for an equitable partition $\mu$ in $X$ is a directed multi-graph in which vertices are parts of $\mu$. $X/ \mu$ has $n$ arcs from $\mu_{j}$ to $\mu_{k}$ if a vertex in $\mu_{j}$ has $n$ neighbours in $\mu_{k}$\cite[Section 2.2]{GMB}. Let $\lambda = [\lambda_{1}, \lambda_{2},\dots, \lambda_{i}]$ be an integer partition. The Young subgroup $\Sym(\lambda) = \Sym(\lambda_{1}) \times \Sym(\lambda_{2}) \times \cdots \times \Sym(\lambda_{i})$ acts on the set of all perfect matchings of $K_{2k}$ and produces an orbit partition on the set of perfect matchings, which is equitable. This means that the quotient graph $X_{\mu} / \lambda$ for the class $\mu$ in the perfect matching association scheme is well-defined; and the eigenvalues of $X_{\mu} / \lambda$ are eigenvalues of $X_{\mu}$. For a complete explanation and more details please see \cite[Section 4]{2int}.\\

In \cite{2int}, we used the Young subgroups to form the quotient  graphs of several matrices $A_{\lambda}$ in $\mathcal{A}$; for the classes $[2k]$, $[2k-2,2]$, $[2k-4,4]$, and $[2k-6,6]$. Hence, we constructed a portion of the character table, by which we were able to calculate the least  and the greatest eigenvalues of the desired weighted adjacency matrix corresponding to the derangement graph defined for $2$-intersecting perfect matchings.
%-----------------------------
\begin{theorem}\cite[p.583]{2int}\label{modules}
Assume that $\Sym(n)$ acts on the set $\Omega$, and that $A$ is the adjacency matrix for an orbital of the action of $\Sym(n)$ on $\Omega$. 
Let $\lambda \vdash n$ and $\pi$ be the orbit partition from the action of $\Sym(\lambda)$ on $\Omega$. If $\eta$ is an eigenvalue of the quotient graph $A/\pi$, then $\eta$ is an eigenvalue of $A$. Moreover, $\eta$ belongs to some $\Sym(n)$-module represented by the partition $\mu$ where $\mu \geq \lambda$ in the dominance ordering.
\end{theorem}
%--------------------------
In this work, we expand these results by calculating the diagonals of several other quotient graphs for the class $[2k-4,2,2]$, corresponding to the young subgroups $\Sym(2k-2)\times \Sym(2)$, $\Sym(2k-4)\times \Sym(4)$, and $\Sym(2k-6)\times \Sym(6)$.\\

Consider the matrix $A_{[2k-4,2,2]}$ in the perfect matching association scheme. In the graph $X_{[2k-4,2,2]}$, two perfect matchings are adjacent if their union forms a 4-cycle and two 2-cycles. Denote the quotient graph of $X_{[2k-4,2,2]}$ corresponding to the group $\Sym(2k-2)\times \Sym(2)$ by $X_{[2k-4,2,2]}/[2k-2,2]$. The adjacency matrix of $X_{[2k-4,2,2]}/[2k-2,2]$ is a $2\times 2$ matrix (Please see the Table \ref{[2k-4,2,2]/[2k-2,2]} in the Appendix \ref{Sec:Appendix1}). The all-ones vector $\bf{1}$ is an eigenvector of this matrix corresponding to the largest eigenvalue, 
$\frac{(2k)!!}{8(2k-4)}$. This eigenvalue is the degree of $A_{[2k-4,2,2]}$, and by Theorem~\ref{modules}, this eigenvalue corresponds to the $[2k]$-module in the character table. It is well-known that the trace of a matrix is equal to the sum of its eigenvalues, so by subtracting the degree from the trace, we find the second eigenvalue of this matrix which is $\frac{(3k-2)}{4}(2k-6)!!$. Using Theorem~\ref{modules}, and noting that the degree eigenvalue belongs to the $[2k]$-module, we deduce that the second eigenvalue belongs to the $[2k-2,2]$-module. Using the same argument for the graphs $X_{[2k-4,2,2]}/[2k-4,4]$, $X_{[2k-4,2,2]}/[2k-4,2,2]$, and  $X_{[2k-4,2,2]}/[2k-6,6]$ respectively, we calculate the eignvalues of the class $[2k-4,2,2]$ for the modules $[2k-2,2]$, $[2k-4,4]$, $[2k-4,2,2]$, and $[2k-6,6]$. The updated character table of the perfect matchings association scheme and some quotient graphs of the class $[2k-4,2,2]$, please see Appendix \ref{Sec:Appendix1}. In table \ref{PartOfCharTable}, the column corresponds to the eigenvalues for the class $[2k-4,2,2]$; and rows correspond to the modules.\\

\FloatBarrier
\begin{center}
\begin{tabular}[h]{N | c | c |}
\hline
\rule{0pt}{20pt} & ~ & $[2k-4,2,2]$ \\
\Xhline{6\arrayrulewidth}
\rule{0pt}{20pt}& $\chi_{[2k]}$ & $\frac{(2k)!!}{8(2k-4)}$ \\
\hline
 
\rule{0pt}{20pt}& $\chi_{[2k-2,2]}$ & \textcolor{RubineRed}{$\frac{(3k-2)(2k-6)!!}{4}$} \\
  
\rule{0pt}{20pt}& $\chi_{[2k-4,4]}$ & \textcolor{RubineRed}{$\frac{-(k+3)(2k-8)!!}{4}$} \\
 
\rule{0pt}{20pt}& $\chi_{[2k-4,2,2]}$ & \textcolor{RubineRed}{$(k^{2}-7k+12)(2k-10)!!$} \\

\rule{0pt}{20pt}& $\chi_{[2k-6,6]}$ & \textcolor{RubineRed}{$\frac{-3}{2}(13k^{2}-101k+190)(2k-12)!!$} \\

\rule{0pt}{20pt}&  \vdots & \vdots \\
\hline
\end{tabular}
\vspace*{0.2cm}
\captionof{table}{$[2k-4,2,2]$-column in the character table of the perfect matching association scheme}
\label{PartOfCharTable}
\end{center}
\FloatBarrier
%______________________________________________________________%
%______________________________________________________________%
\section{Families of Set-wise $2$-intersecting perfect matchings of the maximum size}\label{sec:EKR}
In this section, we determine an appropriate set of coefficients for $B_{2}(2k)$, so that the ratio bound holds with equality. In other words, for $t=2$, the set of all perfect matchings with fixed $2t$ elements in exactly $t$ edges are the maximum families of set-wise $t$-intersecting perfect matchings.  First, we need some background and results concerning the dimensions of the irreducible modules of $\Sym(n)$.\\

Suppose $\mu = [\mu_{1},\mu_{2},\dots, 
\mu_{f}]$ and $\lambda = [\lambda_{1}, \lambda_{2},\dots, \lambda_{g}]$ are two integer partitions of $2k$. If there exists $i\in \{1,2,\dots, min\{f,g\}\}$ such that $\mu_{i}>\lambda_{i}$ and for all $j<i$, $\mu_{j}=\lambda_{j}$, we say $\mu\geq \lambda$ in the \textsl{dominance} ordering. The \textsl{dual} partition of $\lambda$; presented by $\lambda^{*}$; is the partition for which the Young diagram is the reflection of the Young diagram in $\lambda$. It is known that the dimension of a $\lambda$-module and its dual's are the same; which is shown by $m(\lambda) = m(\lambda^{*})$. If $\lambda\geq \lambda^{*}$, then $\lambda$ is called \textsl{primary}.\cite{Ra}. The next two theorems provide bounds on the dimension of some modules.
%-------------------------------------------------
\begin{theorem}\cite[p.163]{Ra}\label{F(n)Theorem}
Let $\lambda$ be a primary partition of $n$ for which the first part $\lambda_{1}< \lfloor \frac{n}{2} \rfloor$. Then $m(\lambda)\geq F(n)$, where
\[
F(n) =
\begin{cases}
  n\cdot F(n-1)(m+2)   &  \textrm{ if $n=2m+1$ is odd},\\
  2\cdot F(n-1)      & \textrm{ if $n$ is even},
\end{cases}
\]
with $F(0) = 2$. In particular, for $n\geq 8$,
\begin{equation}\label{F(n)}
\frac{3}{2}\cdot F(n-1)\leq F(n) \leq 2\cdot F(n-1).
\end{equation}
\end{theorem}
%-------------------------------------------------
\begin{theorem}\cite[p.151]{Ra}\label{RaTheorem}
Let $\lambda = [ \lambda_{1},\lambda_{2},\cdots, \lambda_{t}]$ be an integer partition of $2k$ in which $\lambda_{1} \geq k$. Then,
\begin{equation*}
m([\lambda_{1}, 2k-\lambda_{1}])\leq m(\lambda).
\end{equation*}
\end{theorem}

%-------------------------------------------------
%----------------------------------------- 
\subsection{Set-wise $2$-intersecting perfect matchings}\label{setwise2int}
If we can find a set of coefficients for the graph $B_{2}(2k)$ in (\ref{eq:W.AdjMat}), such that the row sum and the least eigenvalue are $\frac{(2k-1)(2k-3)}{3}-1$ and $-1$ respectively, then the ratio bound will hold with equality,

\[
\mid \mathcal{S}_{2}(2k) \mid = \frac{|V(X)|}{1-\frac{d}{\tau}} = \frac{(2k-1)!!}{1-\frac{\frac{(2k-1)(2k-3)-1}{3}}{-1}} = 3(2k-5)!!.
\]

For $k\geq 3$, set all coefficients $a_{\lambda}$ in $B_{2}(2k)$ to be $0$, except $a_{[2k]}$ and $a_{[2k-2,2]}$. Then we have

\[
\hat{B}_{2}(2k) = a_{[2k]}A_{[2k]}+a_{[2k-2,2]}A_{[2k-2,2]}.
\] 
Using the general formulas for the eigenvalues of $A_{[2k]}$ and $A_{[2k-2,2]}$ in the character table \ref{CharTable}, we construct a system of linear equations in which equations correspond to the irreducible modules $[2k]$, $[2k-2,2]$, and $[2k-4,4]$,

\newcolumntype{C}{>{{}}c<{{}}}
\[
\setlength\arraycolsep{0pt}
\renewcommand\arraystretch{1.25}
\begin{array}{*{3}{rC}l}
   (2k-2)!!\mathbf{a_{[2k]}} & + &  (k)(2k-4)!!\mathbf{a_{[2k-2,2]}} & = & \frac{(2k-1)(2k-3)}{3}-1, \\
   -(2k-4)!!\mathbf{a_{[2k]}} & + &  \frac{1}{2}(2k-4)!!\mathbf{a_{[2k-2,2]}} & = & -1, \\
   -(2k-6)!!\mathbf{a_{[2k]}} & - &  (5k-12)(2k-8)!!\mathbf{a_{[2k-2,2]}} & = & -1.
\end{array}
\]
By considering the second and the third equations above, we obtain the unique solutions $a_{[2k]} = \frac{k}{3(2k-4)!!}$ and $a_{[2k-2,2]} = \frac{(2k-6)}{3(2k-4)!!}$, for all $k\geq 4$;  which satisfies the first equation above. Thus, the row sum of $\hat{B}_{2}(2k)$ would be $d_{\hat{B}} = \frac{(2k-1)(2k-3)}{3}-1$. To finish this argument, we prove that every other eigenvalue of $\hat{B}_{2}(2k)$ is between $-1$ and $d_{\hat{B}}$.
%---------------------------------------
\begin{theorem}\label{LeastEvalTrick2t}
For $k\geq 7$, let 
\[
\hat{B}_{2}(2k) = \frac{k}{3(2k-4)!!}A_{[2k]}+\frac{(2k-6)}{3(2k-4)!!}A_{[2k-2,2]}.
\]
Then the row sum and the least eigenvalue of the matrix $\hat{B}_{2}(2k)$ are $\frac{(2k-1)(2k-3)}{3}-1$ and $-1$, respectively. Furthermore, the only modules with eigenvalue equal to -1 are $[2k-2,2]$ and $[2k-4,4]$, and all other eigenvalues are in $(-1,\frac{(2k-1)(2k-3)}{3}-1 )$.
\end{theorem}
%----------------------------------------- 
\begin{proof}
In \cite{Sri1}, Srinivasan implemented a Maple program to compute the complete character table of the perfect matching association scheme for all $k\leq 40$. So by utilizing the complete character table for $3\leq k\leq 12$, we find the eigenvalues of $\hat{B}_{2}(2k)$, which will verify that the equality holds in the ratio bound. For the remainder of the proof, let $k\geq 12$. Let $\{d_{\hat{B}}^{(1)}, -1^{(m_{1})}, -1^{(m_{2})},\theta_{3}^{(m_{3})}, \dots,  \theta_{\ell}^{(m_{\ell})} \}$ be the spectrum of the matrix $\hat{B}_{2}(2k)$, where the values $m_{i}$ are the multiplicities of the eigenvalues.

It is known that the multiplicity of the for any module $[2k-\ell,\ell]$; say $m([2k-\ell, \ell])$; can be calculated by $m([2k-\ell, \ell]) = \binom{2k}{\ell}-\binom{2k}{\ell-1}$ \cite[Section 12.6]{GMB}. Hence, 
\begin{align*}
m_{1} =& \frac{2k(2k-3)}{2},\\
m_{2} =& \frac{2k(2k-1)(2k-2)(2k-7)}{4!}.
\end{align*}
Now consider the row sum of the matrix $(\hat{B)}^{2}$. The main diagonal entries of $(\hat{B)}^{2}$ are given by
\begin{align*}
\hat{B}^{2}(i,i) = d_{\hat{B}^{2}} =& a_{[2k]}^{2}d_{[2k]}+a_{[2k-2,2]}^{2}d_{[2k-2,2]}\\
=& \frac{k^{2}}{9(2k-4)!!^{2}}(2k-2)!!+\frac{(2k-6)^{2}}{9(2k-4)!!^{2}}k(2k-4)!!\\ 
=& \frac{k(6k^{2}-26k+36)}{9(2k-4)!!},
\end{align*}
where $d_{[2k]}$ and $d_{[2k-2,2]}$ are the degrees of the matrices $A_{[2k]}$ and $A_{[2k-2,2]}$, respectively.
It is also well-known that the trace of a matrix is the sum of its eigenvalues. So,
\begin{align}
\left( \frac{k(6k^{2}-26k+36)}{9(2k-4)!!} \right) (2k-1)!! \nonumber = d_{\hat{B}}^{2}+m_{1}+m_{2}+\sum_{i=3}^{k}m_{i}\theta_{i}^{2}.
\end{align}
Hence,
\begin{align*}
&k\left(6k^{2}-26k+36\right)\frac{(2k-1)!!}{9(2k-4)!!}- \left(\frac{(2k-1)(2k-3)-3}{3}\right)^{2} \\
&- \frac{2k(2k-3)}{2}-\frac{2k(2k-1)(2k-2)(2k-7)}{4!}=\sum_{i=3}^{k}m_{i}\theta_{i}^{2}.
\end{align*}
This means that for every $\theta_{i}$, $3\leq i \leq \ell$, we have that,
\begin{equation}\label{ineq00}
k\left(6k^{2}-26k+36\right)\frac{(2k-1)!!}{9(2k-4)!!}-\frac{k(11k-25)(2k-1)(2k-3)}{18} \geq m_{i}\theta_{i}^{2}.
\end{equation}

Next we show that any eigenvalue $\theta_{i}$ is greater than $-1$, for $3\leq i \leq \ell$. Let $\theta_{3}$ be the eigenvalue of $\hat{B_{2}(2k)}$ corresponding to the module $[2k-4,2,2]$. Then, $\theta_{2}$ is the linear combination of the eigevalues of the matrices $A_{[2k]}$ and $A_{[2k-2,2]}$ in the character table \ref{CharTable}, for the same module. 
\begin{equation*}
\theta_{3} = \frac{k}{3(2k-4)!!}(2(2k-6)!!)+\frac{(2k-6)}{3(2k-4)!!}(-(2k-6)!!)= \frac{1}{(k-2)}.
\end{equation*}
For the other eigenvalues, it is sufficient to show the following inequality holds.
\begin{equation}\label{eq:EvalTrick4}
k\left(6k^{2}-26k+36\right)\frac{(2k-1)!!}{9(2k-4)!!}-\frac{k(11k-25)(2k-1)(2k-3)}{18} < m_{i},
\end{equation}
since this, along with the equation (\ref{ineq00}) shows that $\theta_{i}^{2}<1$.
Let $m_{i}$ be the multiplicity of the $\lambda_{i}$-module, where $\lambda_{i} = [\lambda_{i_1}, \lambda_{i_2},\cdots, \lambda_{i_\ell}]$, and $\lambda_{i_1}\geq \lambda_{i_2} \geq \cdots \geq \lambda_{i_\ell}$. There are $2$ cases:
\begin{itemize}
\item[{\bf Case 1.}] Suppose that $\lambda_{i_1}\geq k$.

Consider the module $[2k-6,6]$. Then,

\[
m_{i} = \binom{2k}{6}-\binom{2k}{5} = \frac{(2k)(2k-1)(2k-2)(2k-3)(2k-4)(2k-11)}{6!}.
\]
By substituting the formula for $m_{i}$ in inequality (\ref{eq:EvalTrick4}), and approximating the term $(2k-4)!!$ by $(2k-5)!!$, the inequality holds for all $k\geq 34$ (this can be confirmed with Maple software). Also, using Maple for the values of $12\leq k\leq 33$, we can see that inequality (\ref{eq:EvalTrick4}) holds. By Theorem \ref{RaTheorem}, and the fact that $m([2k-\ell, \ell]) = \binom{2k}{\ell}-\binom{2k}{\ell-1}$, the multiplicity $m_{i}$ is greater than or equal to the multiplicity of the module $[2k-6,6]$, for all modules with $k \leq \lambda_{i_1}\leq 2k-6$. Hence, inequality (\ref{eq:EvalTrick4}) holds for all modules with $\lambda_{i_{1}}>k$.
\item[{\bf Case 2.}] Suppose that $\lambda_{i_1}< k$. 

If $\lambda_{i}$ is primary, by Theorem \ref{F(n)Theorem}, For $2k\geq 8$, we have that,
\begin{equation}\label{UpBound}
m(\lambda)\geq F(2k) = 2F(2k-1) \geq (2)     \left( \frac{3}{2} \right) F(2k-2)\geq \cdots \geq 3^{k-4}F(8)\geq 403200(3^{k-4}).
\end{equation}
Using (\ref{UpBound}), and approximating the term $(2k-4)!!$ by $(2k-5!!)$ in (\ref{eq:EvalTrick4}), it is sufficient to show that,
\[
\frac{48k^{5}-348k^{4}+928k^{3}-965k^{2}+921k}{18}<403200(3^{k-4}).
\]
The expression $-348k^{4}+928k^{3}-965k^{2}+921k$ in the numerator of the fraction on the left side of the above inequality is negative for all $k\geq 2$, so it is sufficient to check when $\frac{48k^{5}}{18}<403200(3^{k-4})$. In fact, this inequality is true for all $k$; therefore, for all primary partitions $\lambda_{i}$ with $\lambda-{i_1}<k$, the in equality (\ref{eq:EvalTrick4}) holds.

Now assume that $\lambda_{i_1}< k$ and $\lambda_{i}$ is not primary. The proof of this part follows from the proof of Case 3 in Theorem 4.11 in \cite{2int}. The dual of $\lambda_{i}$; $\lambda^{*} = (\lambda_{1}^{*},\lambda_{2}^{*},\cdots, \lambda_{t}^{*})$; is primary and $\lambda_{1}^{*}=\lambda_{2}^{*}$. Note that $m(\lambda_{i}) = m_{\lambda^{*}}$. If $\lambda_{1}^{*}\geq k$, then $\lambda^{*}$ is $[k,k]$, which is covered  by Case 1. For $\lambda_{1}^{*}<k$, we just proved the result in Case 2.
\end{itemize}

\end{proof}
%-----------------------------------------
\section{Conjecture on Set-wise $3$-intersecting perfect matchings}\label{sec:set-wise3int}
The approach we take in this section is the same as the one we took in \ref{setwise2int} for set-wise $2$-intersecting perfect matchings. For $t=3$, if we can find a set of coefficients in (\ref{eq:W.AdjMat}), so that the weighted adjacency matrix has degree $d=\frac{(2k-1)(2k-3)(2k-5)}{15}-1$ and least eigenvalue  $\tau = -1$, then $1-\frac{d}{\tau}=\frac{\mid V(N_{3}(2k)) \mid}{\mid \mathcal{S}_{3}(2k) \mid}$; thus, the ratio bound holds with equality. Let
\[
\hat{B}_{3}(2k) = a_{[2k]}A_{[2k]}+a_{[2k-2,2]}A_{[2k-2,2]}+a_{[2k-4,2,2]}A_{[2k-4,2,2]}.
\] 
We set a system of linear equations for $\hat{B}_{3}(2k)$; similar to the one for $t=2$. In this system equations correspond to the irreducible modules $[2k-2,2]$, $[2k-4,4]$, and $[2k-6,6]$; we want weights so the corresponding eigenvalues equal to $-1$. 

\newcolumntype{C}{>{{}}c<{{}}}
\[
\setlength\arraycolsep{0pt}
\renewcommand\arraystretch{1.25}
\begin{array}{*{3}{rC}l}
\frac{(2k)!!}{2k}\mathbf{a_{[2k]}}& + & \frac{(2k)!!}{2(2k-2)}\mathbf{a_{[2k-2,2]}}& + & \frac{(2k)!!}{8(2k-4)}\mathbf{a_{[2k-4,2,2]}} & = & d, \\

-(2k-4)!!\mathbf{a_{[2k]}}& + &  \frac{(2k-4)!!}{2}\mathbf{a_{[2k-2,2]}}& + & \frac{(3k-2)(2k-6)!!}{4}\mathbf{a_{[2k-4,2,2]}} & = & -1, \\

-(2k-6)!!\mathbf{a_{[2k]}}& - &  (5k-12)(2k-8)!!\mathbf{a_{[2k-2,2]}}& - & \frac{(k+3)(2k-8)!!}{4}\mathbf{a_{[2k-4,2,2]}} & = & -1, \\

-3(2k-8)!!\mathbf{a_{[2k]}}& - &  (3k-10)(2k-10)!!\mathbf{a_{[2k-2,2]}}& - & \frac{-3(13k^{2}-101k+190)(2k-12)!!}{2}\mathbf{a_{[2k-4,2,2]}} & = & -1, \\
\end{array}
\]

For all $k\geq 6$, the unique solutions of this system are $\mathbf{a_{[2k]}} = \frac{(k-3)(7k-10)}{30(2k-4)!!}$, $\mathbf{a_{[2k-2,2]}} = \frac{-2(k^{2}-10k+15)}{15(2k-4)!!}$, and $\mathbf{a_{[2k-4,2,2]}} = \frac{2(k-5)}{5(2k-6)!!}$ . 

\begin{conj}\label{LeastEvalTrick3t}
For $k\geq 11$, let 
\[
\hat{B}_{3}(2k) = A_{[2k]}\frac{(k-3)(7k-10)}{30(2k-4)!!}+A_{[2k-2,2]}\frac{-2(k^{2}-10k+15)}{15(2k-4)!!}+A_{[2k-4,2,2]}\frac{2(k-5)}{5(2k-6)!!}.
\]
Then the row sum and the least eigenvalue of the matrix $\hat{B}_{3}(2k)$ are $\frac{(2k-1)(2k-3)(2k-5)}{15}-1$ and $-1$, respectively. Furthermore, the only modules with eigenvalue equal to -1 are $[2k-2,2]$, $[2k-4,4]$, and $[2k-6,6]$.
\end{conj}
%---------------------------------------
Similar to the proof of Theorem \ref{LeastEvalTrick2t}, for all modules except $[2k-6,4,2]$ and $[2k-6,2,2,2]$, it can be shown that the associated eigenvalues are between $\frac{(2k-1)(2k-3)(2k-5)}{15}-1$ and $-1$. We guess that this is true for modules $[2k-6,4,2]$ and $[2k-6,2,2,2]$; as we confirmed this with Maple \cite{Sri1}; but finding the exact formulas of the corresponding eigenvalues of this modules using the quotient graphs is quite complicate. Using symmetric functions and character theory might be an approaches for more broad results in the future \cite{Mu}.
\section{Further Work}
In this paper we proved that the Erd\H{o}s-Ko-Rado theorem holds for set-wise $2$-intersecting perfect matchings of the complete graph $K_{2k}$. The very first question that arises is if we can extend our results to set-wise $t$-intersecting perfect matchings of the complete graph $K_{2k}$ with $3\leq t \leq \frac{k}{2}$. In Section \ref{setwise2int}, we defined a weighted adjacency for $t=3$ and conjectured for which the ratio bound holds with equality, so starting from set-wise $3$-intersecting could be a very good start to do further investigating on this interesting problem.

Also, while we were working with the character tables of the perfect matchings association scheme, we found the following patterns:

\begin{conj}
The eigenvalues of the class $[2k]$ corresponding to the modules $[2k-2i,2i]$ and $[2k-2i,2,\cdots,2]$ are $-(2i-3)!!(2k-2i-2)!!$ and $(-1)^{i}(i!)(2k-2i-2)!!$, respectively.
\end{conj}
Note that these two eigenvalues have high multiplicities, so I think it is worthwhile to derive their formula.

%______________________________________________________________%
%______________________________________________________________%
\section{Acknowledgements}
The author wish to thank Shaun Fallat and Karen Meagher for the careful reading of the manuscript, and
for their insightful comments. 
\bibliographystyle{plain}

%===============================================================
%===============================================================
%===============================================================
\newpage
\appendix
\section{Diagonal Entries of The Adjacency Matrices of Some Quotient graphs in the Perfect Matching Association Scheme}\label{Sec:Appendix1}
%===============================================================

\begin{centering}
\begin{tabular}{N|M{2.7cm}|M{2.7cm}|}
\hline 
\vspace*{0.6cm}
\rule{0pt}{10pt}& $(k-1)(2k-6)!!$ & $~$ \\ [1ex] 
\hline
\vspace*{0.6cm}
\rule{0pt}{10pt}& $~$ & $\frac{1}{4}(2k^{2}+k-2)(2k-6)!!$ \\[1ex] 
\hline
\end{tabular}
\vspace*{0.2cm}
\captionof{table}{The adjacency matrix of $X_{[2k-4,2,2]}/[2k-2,2]$}
\label{[2k-4,2,2]/[2k-2,2]}
\end{centering}

%===============================================================
\begin{centering}
%\begin{tabular}{|c| c| c| c| c| c?l|}
\begin{tabular}{N| M{3.5cm} | M{3.9cm} | M{3.9cm} |}
\hline 
\vspace*{0.6cm}
\rule{0pt}{10pt}& $(2k-6)!!$ & $~$ & $~$ \\ [1ex] 
\hline
\vspace*{0.6cm}
\rule{0pt}{10pt}& $~$ & $\frac{1}{4}(9k-20)(2k-6)!!$ & $~$\\[1ex] 
\hline 
\vspace*{0.6cm}
\rule{0pt}{10pt}& $~$& $~$ & $(k^{3}-7k^{2}+\frac{75}{4}k-\frac{87}{4})(2k-8)!!$ \\[1ex] 
\hline 
\end{tabular}
\vspace*{0.2cm}
\captionof{table}{The adjacency matrix of $X_{[2k-4,2,2]}/[2k-4,4]$}
\end{centering}

%===============================================================
\begin{centering}
\begin{tabular}{N| M{2cm} | M{3cm} | M{3.7cm} | M{3.7cm} |}
\hline 
\vspace*{0.6cm}
\rule{0pt}{15pt}& $\bf{0}$ & $~$ & $~$ & $~$ \\ [1ex] 
\hline
\vspace*{0.6cm}
\rule{0pt}{15pt}& $~$ & $(20k-78)(2k-8)!!$ & $~$ & $~$ \\[1ex] 
\hline 
\vspace*{0.6cm}
\rule{0pt}{15pt}& $~$ & $~$ & $(18k^{3}-221k^{2}+953k-1455)(2k-10)!!$ & $~$\\[1ex] 
\hline 
\vspace*{0.7cm}
\rule{0pt}{15pt}& $~$ & $~$ & $~$ & $\frac{1}{2}(2k-11)(4k^{4}-60k^{3}+371k^{2}-1155kk+1530)(2k-12)!!$\\[1ex] 
\hline 
\end{tabular}
\vspace*{0.2cm}
\captionof{table}{The adjacency matrix of $X_{[2k-4,2,2]}/[2k-6,6]$}
\end{centering}

%-------------------------------------------------------------
%===============================================================
\appendix
\subsection{Character Table of the Perfect Matching Association Scheme}
%===============================================================

\FloatBarrier
\begin{sidewaystable}
\begin{centering}
\begin{tabular}[h]{N| m{1.25cm} ? m{1.75cm} | m{3.25cm} | m{4.35cm} | c | m{2.8cm}|}
\hline
\rule{0pt}{20pt}& ~ & $[2k]$ & $[2k-2,2]$ & $[2k-4,4]$ & $[2k-4,2,2]$ & $[2k-6,6]$ \\[4ex]
\Xhline{5\arrayrulewidth}
\rule{0pt}{20pt}& $\mathbf{\chi}_{[2k]}$ & $\frac{(2k)!!}{2k}$ & $\frac{(2k)!!}{2(2k-2)}$ & $\frac{(2k)!!}{4(2k-4)}$ & $\frac{(2k)!!}{8(2k-4)}$ & $\frac{(2k)!!}{6(2k-6)}$ \\[4ex]
\hline
 
\rule{0pt}{20pt}& $\chi_{[2k-2,2]}$ & $-(2k-4)!!$ & $\frac{(2k-4)!!}{2}$ & $\frac{-2k(2k-6)!!}{4}$ & \textcolor{RubineRed}{$\frac{(3k-2)(2k-6)!!}{4}$} & $\frac{-2k(2k-4)!!}{6(2k-6)}$ \\[4ex]
  
\rule{0pt}{20pt}& $\chi_{[2k-4,4]}$ & $-(2k-6)!!$ & $-(5k-12)(2k-8)!!$ & $\frac{(7k-15)(2k-8)!!}{2}$ & \textcolor{RubineRed}{$\frac{-(k+3)(2k-8)!!}{4}$} & $\frac{-2k(2k-6)!!}{6(2k-6)}$ \\[4ex]
 
\rule{0pt}{20pt}& $\chi_{[2k-4,2,2]}$ & $2(2k-6)!!$ & $-(2k-6)!!$ & $\frac{-(2k-6)!!}{2}$& \textcolor{RubineRed}{$(k^{2}-7k+12)(2k-10)!!$} & $\frac{4k(2k-6)!!}{6(2k-6)}$\\[4ex]

\rule{0pt}{20pt}& $\chi_{[2k-6,6]}$ & $-3(2k-8)!!$ & $-3(3k-10)(2k-10)!!$ & $-3(9k^{2}-71k+140)(2k-12)!!$ & \textcolor{RubineRed}{$\frac{-3}{2}(13k^{2}-101k+190)(2k-12)!!$} & $6(5k^{2}-38k+70)(2k-12)!!$\\[4ex]

\rule{0pt}{20pt}&  \vdots & \vdots & \vdots & \vdots & \vdots & \vdots \\[4ex]
\hline
\end{tabular}
\caption{Character table of the perfect matching association scheme}
\label{CharTable}
\end{centering}
\end{sidewaystable}
\FloatBarrier

\end{document}